\documentclass[12pt,a4paper,oneside,openright]{amsart}

\usepackage{xcolor}
\usepackage[english]{babel}
\usepackage[T1]{fontenc}
\usepackage[utf8]{inputenc}
\usepackage{amsmath,amsfonts,amssymb,amsthm}
\usepackage{tikz}
\usetikzlibrary{arrows,decorations.pathmorphing,positioning,backgrounds,fit,matrix}
\newtheorem{thm}{Theorem}[section]
\newtheorem{lem}[thm]{Lemma}
\newtheorem{prop}[thm]{Proposition}

\newtheorem{defn}[thm]{Definition}

\newtheorem{exmp}[thm]{Example}

\newtheorem{remark}[thm]{Remark}

\newtheorem{cor}[thm]{Corollary}
\newtheorem{theoremA}{Theorem}

\begin{document}
\title[Curvatures and Volume of Graphs]{Inner-Outer Curvatures, Ricci-Ollivier Curvature and Volume Growth  of Graphs}
\author{Andrea Adriani and Alberto G. Setti}

\maketitle

\begin{abstract}

We are concerned with the study of different notions of curvature on graphs. We show that if a graph has stronger inner-outer curvature growth than a model graph, then it has faster volume growth too. We also  study the relationhips of volume growth with other kind of curvatures, such as the Ollivier-Ricci curvature.
\\

\noindent
DiSTA,   Universit\'a dell'Insubria,  Via Valleggio 11, 22100 Como, Italy\\
 (aadriani@uninsubria.it)\\
DiSAT,   Universit\'a dell'Insubria Insubria,  Via Valleggio 11, 22100 Como, Italy\\
 (alberto.setti@uninsubria.it)\\
\end{abstract}

\section{Introduction and main results}

In recent times there has been an increasing interest in the study of different notions of curvature on graphs (see, for example \cite{{BJL12}, {HL}, {HM}, {LY10}, {LLY11}, {LLY13}, {LMP}, {LV}, {MW}, {Oll07}, {Oll09}, {Sturm}}) in order to obtain analogues  of well known results valid for Riemannian manifolds. In many cases this study leads to results very different from those obtained on manifolds, unveiling intriguing differences between these two realms and confirming the intrinsic interest of the subject (see, for example, \cite{MW}, Theorem 4.11). It is well known that the volume growth of a Riemannian manifold can be controlled from above in terms of its Ricci curvature: by way of example, if a Riemannian manifold has a greater Ricci curvature than a model manifold, then it has smaller volume growth (by contrast, a control from below is significantly more delicate and involves both the sectional curvature and topological properties of the space).  It is natural to ask if, using appropriate notions of curvature, such results have  analogues in the graph setting. The situation appears to be less straightforward than one may think. On the one hand, an outcome of our investigation is that, using  notation and terminology as in \cite{KLW}, a stronger curvature growth  in general implies faster volume growth (the difference in the versus of the inequality is to be ascribed to a change of sign in the definition of the curvature). On the other hand it turns out that, contrary to what happens in the manifolds setting,
a control of the Laplacian of the distance function (which is implied by a control of the Ollivier Ricci curvature, \cite{MW}, and in turn allows to obtain gradient estimates) does not in general allow to control volume growth, thus showing a significant difference  with what happens in the manifold setting.

The paper is organized as follows: in Section 2 we introduce the setting and some basic notation on graphs with the different notions of curvature considered. In Section 3 we state and prove our main results. The following special case of Theorem~\ref{comparison} below illustrates how controlling the inner/outer curvatures allows to control volume growth.
\begin{theoremA}
Let $ G_1 =(V_1,b_1,m_1) $ and $ G_2=(V_2,b_2,m_2) $ with roots $o_i\in G_i$ and  inner/outer curvatures
 $k_\pm^i$, $i=1,2$.
If $k_+^1(x_1)\geq k_+^2(x_2)$ and $k_-^1(x_1)\leq k_-^2(x_2)$ for all $x_1\in S_r(o_1)$, $x_2\in S_r(o_2)$ and all $r\geq 0$ then $m_1(S_r(o_1))\geq m_2(S_r(o_2))$.
\end{theoremA}

While the conclusion of the theorem does not hold if it is only assumed that the curvatures inequalities are satisfied for large enough $r$,  we are able to show that an asymptotic control on the curvature allows to control the volume growth up to a constant, see Theorem~\ref{asympcomp}.

We next study the relationships between the Ollivier curvature of a graph (for some useful results on the subject see, for example, \cite{MW}) and its inner and outer curvatures, showing that, in general, this notion is not strong enough to determine the behaviour of the volume growth of the graph for comparison theorems, not even in the case of model graphs, or simpler birth-death chains.

By contrast, the positivity of the Ollivier Ricci curvature implies bounds on the diameter of a graph and therefore of its volume under assumptions on the degree (\cite{BRT,MW,Pae}). In addition, the Ollivier curvature has shown to have a role in obtaining spectral estimates (\cite{BJL12,BRT}).

\section{Set up and notation}

A graph is a quadruple $ G = (V,b,c,m) $, where $ V $ is  a countable set,  $ m: V \to (0,\infty) $ is a measure of full support on $V, $
$ b: V \times V \to [0, \infty) $ is a symmetric function which vanishes on the diagonal and represents the edge weight and $ c : V \to [0,\infty) $ is the potential, or killing term, of the graph. We say that two vertices $ x $ and $ y $ are neighbors, and write $ x \sim y $, if $ b(x,y)> 0 $; in this case we denote by $(x,y)$ the edge connecting $ x $ and $ y $. A path in $V$ is a sequence of vertices  $\dots \sim x_{i-1}\sim x_i\sim x_{i+1}\sim \dots $, and a graph is connected if for every $ x, y \in  V $ there exists             a path $ x_0=x \sim x_1 \sim ... \sim x_n = y $ joining $x$ and $y$. In this case, the number $n$ of edges in the path is the (combinatorial) length of the path and the distance $d(x,y)$ between $x$ and $y$ is the length of the shortest path connecting $x$ and $y$. Further, we say that a graph is locally finite if every $ x \in V $ has finitely many neighbors, that is if $ \left| \left\{ y: b(x,y) > 0 \right\} \right| < \infty $. Note that in particular this condition implies that the degree
$$
\mathrm{Deg}(x)=\frac 1{m(x)} \sum_{y \in V} b(x,y)
$$
is finite for every $ x \in V $.\\
In this paper we will consider locally finite graphs with no killing term. We will then denote a generic graph by the triple $ (V,b,m) $.\\
We let $ C(V) = \left\{ f : V \to \mathbb{R} \right\} $ denote the space of real-valued functions on $ V $ and define the formal Laplacian $ \Delta : C(V) \to C(V) $ by the formula
\begin{equation*}
\Delta f(x) = \frac{1}{m(x)} \sum_{y \in V} b(x,y) \left( f(x)-f(y) \right),
\end{equation*}
for all $ x \in V $.\\
Note that, because of the assumption of local finiteness, the formal Laplacian is well-defined for every function $ f \in C(V) $ and for every $ x \in V $.
\\
Let $ x_0 \in V $ be a fixed vertex.
For every nonnegative integer $r\in \mathbb{N}_{0}$, we write
$S_r:=S_r(x_0):= \left\{ x : d(x,x_0) = r \right\} $ and $ B_r:=B_r(x_0):= \left\{ x : d(x,x_0) \leq r \right\} $ and define the inner and outer curvatures at   $ x \in S_r $ as
\begin{equation*}
k_{\pm}(x) = \frac{1}{m(x)} \sum_{y \in S_{r \pm 1}} b(x,y), \quad k_-(x_0)=0.
\end{equation*}
We say that a graph is weakly spherically symmetric, or that it is a model, if for some  vertex $ o $ (which we will refer to as the  root of the graph) the corresponding inner and outer curvatures  $ k_{\pm} $ are spherically symmetric functions, that is if $ k_{\pm}(x) = k_{\pm}(x') $ for every $ x, x' \in S_r(o) $, for every $ r \geq 0 $. Moreover, if $ V = \mathbb{N}_0 $ and $ b(x,y) = 0 $ whenever $ |x-y| \neq 1 $  we say that $ G $ is a birth-death chain.\\

In the context of metric measure spaces, other useful notions of curvature have been defined in terms of transport theory, see \cite{LV,Ohta,Oll07,Oll09, Sturm}. In the setting of graph theory a particularly fruitful choice is the Ricci curvature introduced in \cite{LLY11}, and later extended in \cite{MW}, where it is referred to as Ollivier (Ricci) curvature and defined by
\begin{equation*}
k(x,y) = \lim_{\epsilon \to 0} \left( 1 - \frac{W(m_{x}^{\epsilon}, m_{y}^{\epsilon})}{d(x,y)} \right),
\end{equation*}
for every couple of vertices $ x, y \in V $, with $ x \neq y $,
where $ W $ denotes the $ L^1 $-Wasserstein distance and
\begin{equation*}
m_x^{\epsilon}(y) = \begin{cases} 1-\epsilon \text{Deg}(x) & \text{ : } y=x \\
  \epsilon b(x,y)/m(x) & \text{ : otherwise.}
\end{cases}
\end{equation*}
A crucial result in \cite{MW} is that the Ollivier curvature can be equivalently defined in terms  of the Laplacian, namely,
\begin{equation}\label{Oll curv}
k(x,y) = \inf\limits_{f \in \text{Lip}(1), \nabla_{xy}f=1} \nabla_{xy} \Delta f,
\end{equation}
where
\begin{equation*}
\nabla_{xy} f= \frac{f(y)-f(x)}{d(x,y)}
\end{equation*}
and Lip$(1) = \left\{ f \in C(V) : |f(x)-f(y)| \leq d(x,y) \text{ for all } x,y \in V \right\} $.

Bounds on the Ollivier Ricci curvature have been used in \cite{MW} to obtain estimates for the Laplacian of the distance function, to describe optimal conditions for the stochastic completeness of a weighted graph and to deduce diameter bounds. On the other hand, we will show in Subsection~\ref{OllCurv} that in general the Ollivier curvature does not allow to control the volume growth.

\section{Curvatures and Volume on graphs}

In this section we study the relationships between inner and outer curvatures and volume for model graphs and general graphs. We begin by defining what it means for a graph to have stronger/weaker curvature growth than that of a model (see \cite{KLW,RKW2}, where this notion is used to obtain comparison results concerning stochastic properties, such as the Feller property and stochastic completeness).
Moreover, we recall the definition of a birth-death chain associated to a graph, see \cite{MW}, where such graphs are used to obtain Laplacian comparison results.
\begin{defn}\label{strongercurv}
Let $ G=(V,b,m) $ be a graph and $ \tilde{G}=(\tilde{V}, \tilde{b}, \tilde{m}) $ be a model graph with root $ o $. Let $ x_0 \in V $ be a fixed vertex. We say that $ G $ has stronger (respectively, weaker) curvature growth than $ \tilde{G} $ if
\begin{itemize}
\item[(i)] $ m(x_0) = \tilde{m}(o) $,
\item[(ii)] for all $ r \geq 0 $ and  $ x \in S_r(x_0) $,
\begin{equation*}
\begin{split}
\, & k_{+}(x) \geq \tilde{k}_{+}(r) \text{  and  } k_{-}(x) \leq \tilde{k}_{-}(r) \\
\text{(respectively, } & k_{+}(x) \leq \tilde{k}_{+}(r) \text{  and  } k_{-}(x) \geq \tilde{k}_{-}(r)),
\end{split}
\end{equation*}
where we recall that, by definition,  $ k_{-}(x_0) = \tilde{k}_-(o) = 0 $.
\end{itemize}
We say that $ G $ has faster volume growth than $ \tilde{G} $ if, for all $ r \geq 0 $,
\begin{equation*}
m \left( S_r (x_0) \right) \geq \tilde{m} \left( S_r(o) \right)
\end{equation*}
\end{defn}

\begin{defn}
Let $ G =(V,b,m) $ be a graph and let $ x_0 \in V $ be a fixed vertex. Its associated birth-death chain $ \bar{G}=(\mathbb{N}_0, \bar{b}, \bar{m}) $ is defined by setting
\begin{equation*}
\bar{m}(r) := m \left( S_r \right)  \text{ for all }  r \geq 0,
\end{equation*}
\begin{equation*}
\bar{b}(r,r+1) := \sum_{x \in S_r \atop y \in S_{r+1}} b(x,y) \text{ for all } r \geq 0.
\end{equation*}
We remark that the summation on the right hand side in the above formula is precisely the quantity denoted by $ \partial B(r) $ in \cite{KLW}.
\end{defn}

Note that, if $ \mathcal{A}: C(V) \to C(V) $ denotes the averaging operator which acts as
\begin{equation*}
\mathcal{A} f(x) = \frac{1}{m \left( S_r \right)} \sum_{y \in S_r} f(y) m(y)
\end{equation*}
for all $ x \in S_r $, then
\begin{equation*}
\begin{split}
\bar{k}_{\pm}(r) & = \frac{1}{\bar{m}(r)} \bar{b}(r,r \pm 1) \\
& = \frac{1}{m \left( S_r \right)} \sum_{z \in S_r \atop y \in S_{r \pm 1}} b(z,y) \\
& = \frac{1}{m \left( S_r \right)} \sum_{z \in S_r} \frac{1}{m(z)} \sum_{y \in S_{r \pm 1}} b(z,y) m(z) \\
& = \mathcal{A}k_{\pm}(x).
\end{split}
\end{equation*}

This shows that the birth-death chain associated with $G$ has the same volume growth and its inner and outer curvatures are the averages of those of $G$.
Since we are interested in volume growth comparisons and estimates, the above discussion motivates the following definition of \textit{stronger average curvature growth}, which will be used to prove one of the main results of this section.

\begin{defn}\label{strongercurvnew}
Let $ G_1=(V_1,b_1,m_1), G_2=(V_2,b_2,m_2) $ be two graphs. Let $ x_1 \in V_1 $ and $ x_2 \in V_2 $ be fixed vertices. We say that $ G_1 $ has stronger average curvature growth than $ G_2 $ if the birth-death chain $ \bar{G}_1 $   associated with $ G_1 $  has stronger curvature growth than  the birth-death chain $ \bar{G}_2 $ associated with $ G_2 $. Namely, denoting by $ \mathcal{A}_i $ the averaging operator on $ G_i $,
\begin{itemize}
\item[(i)] $ m_1(x_1) = m_2(x_2) $,
\item[(ii)] for all $ r \geq 0 $ and  $ x' \in S_r(x_1), x'' \in S_r(x_2) $,
\begin{equation*}
\begin{split}
\, & \mathcal{A}_1 k^1_{+}(x') \geq \mathcal{A}_2 k^2_{+}(x'') \text{  and  } \mathcal{A}_1 k^1_{-}(x') \leq \mathcal{A}_2 k^2_{-}(x''), \\
\end{split}
\end{equation*}
where, by definition, $ k^1_{-}(x_1) = k^2_-(x_2) = 0 $.
\end{itemize}
\end{defn}

\begin{remark}
It is clear that, if two graphs satisfy Definition \ref{strongercurv} then they also satisfy  Definition \ref{strongercurvnew}. We want to underline that the validity of Definition \ref{strongercurvnew} is not enough to obtain comparison theorems for the usual stochastic properties. By way of  example, in  \cite{XH} the author constructed an example of a graph  $ G $  such that
\begin{equation}\label{stoc_compl}
\sum_{r=0}^{\infty} \frac{m(B_r)}{\displaystyle{\sum_{x \in S_r, y \in S_{r+1}} b(x,y)}} = \infty,
\end{equation}
but fails to satisfy the weak Omori-Yau maximum principle, and therefore it is stochastically incomplete. On the other hand, the birth-death chain $\overline{G}$ associated to $G$  is stochastically complete since it clearly satisfies \eqref{stoc_compl}, which is a necessary and sufficient condition for the stochastic completeness of model graphs (see \cite{KLW}, Theorem 5), while the condition in Definition \ref{strongercurvnew} trivially holds with equality.
\end{remark}

We are now ready to state and prove the main result of this section.

\begin{thm}\label{comparison}
Let $ G_1 =(V_1,b_1,m_1) $ and $ G_2=(V_2,b_2,m_2) $ be two graphs such that $ G_1 $ has stronger average curvature growth than $ G_2 $. Then $ G_1 $ has faster volume growth than $ G_2 $.
\end{thm}

\begin{proof}
By the above discussion, without loss of generality we may assume that $ G_1 $ and $ G_2 $ are birth-death chains.\\
We proceed by induction: for $ r=0 $, by the normalization assumption, we have $ m_1(0) = m_2(0) $.\\
For the induction argument we now assume that $ m_1(r) \geq m_2(r) $ and prove that $ m_1(r+1) \geq m_2(r+1) $.\\
By assumption
\begin{equation*}
 \frac{b_1(r,r+1)}{m_1(r)}=k_+^1(r) \geq k_+^2(r)= \frac{b_2(r,r+1)}{m_2(r)} ,
\end{equation*}
whence, rearranging,
\begin{equation}
\label{dis 2}
1\leq \frac{m_1(r)} {m_2(r) }\leq \frac{b_1(r,r+1)}{b_2(r,r+1)}.
\end{equation}
Moreover, by assumption,
\begin{equation*}
 \frac{b_1(r,r+1)}{m_1(r+1)}=k_-^1(r+1) \leq  k_-^2(r+1)=\frac{b_2(r,r+1)}{m_2(r+1)},
\end{equation*}
so that, using \eqref{dis 2},
\begin{equation*}
m_1(r+1) \geq \frac{b_1(r,r+1)}{b_2(r,r+1)} m_2(r+1) \geq m_2(r+1),
\end{equation*}
as required to complete the proof.
\end{proof}

In the case where one of the two graphs considered in Theorem \ref{comparison} is a model, we get the following immediate corollary.

\begin{cor}
Let $ G =(V,b,m) $ be a graph and $ \tilde{G}=(\tilde{V}, \tilde{b}, \tilde{m} ) $ be a model graph. Assume that $ G $ has stronger (respectively, weaker) curvature growth than $ \tilde{G} $. Then $ G $ has faster (respectively, slower) volume growth than $ \tilde{G} $.
\end{cor}

\begin{remark}
We want to underline the fact that, differently from what happens in the setting of Riemannian manifolds, it is not sufficient to have a comparison assumption on the Laplacian of the distance function to obtain a comparison result concerning volume growth. Indeed, note that the assumption of stronger curvature growth implies the weaker condition
\begin{equation*}
k_+(x) - k_-(x) \geq \tilde{k}_+(r) - \tilde{k}_-(r),
\end{equation*}
which is exactly equivalent to
\begin{equation*}
\Delta d(x_0,x) \leq \tilde{\Delta} d(0,r)
\end{equation*}
for all $ x \in S_r $ and $r\ge 0$.\\
However, such a condition is not enough to guarantee the conclusion of Theorem \ref{comparison} as the following example shows.
\end{remark}

\begin{exmp}\label{exmp 1}
Let $ G  =(V,b,m) $ be the unweighted birth-death chain, that is $ V = \mathbb{N}_0 $, $ m(r) = 1 $ and $ b(r,r+1) = 1 $ for all $ r \geq 0 $, and let $ G' = (\mathbb{N}_0, b',m') $ be a birth-death chain such that $ m'(r) = r+1 $ and $ b'(r,r+1) = (r+1)^{-2} $. It follows that
\begin{align*}
k_+(0) - k_-(0)&=1,\\
k_+(r)-k_-(r)  &\equiv 0 \text{ } \forall r  > 0
\end{align*}
while
\begin{align*}
k'_+(0) - k'_-(0)&=1,\\
k'_+(r)-k'_-(r) &= \frac{1}{(r+1)^3}- \frac{1}{r^2(r+1)}= \frac{-2r-1}{r^2(r+1)^3} < 0 \text{ } \forall r > 0,
\end{align*}
so that $ k'_+(r)-k'_-(r) \leq k_+(r)-k_-(r) $ for all $ r \geq 0 $.
On the other hand, it is clear that $ G' $ has faster volume growth than $ G $.
\end{exmp}

\begin{remark}
We want to stress the fact that, exactly as in the manifold case, in order to get a volume comparison the curvature inequality in the statement of Theorem \ref{comparison} must hold for every $r\geq 0$ and not just for all $r\ge R > 0 $ (see Definition \ref{strongcurvout} and \cite{RKW2}). This is shown in the following example.
\end{remark}

\begin{exmp}
Let $ \bar{G} $ be a birth-death chain and $ \tilde{G} = (\mathbb{Z}, \tilde{b},\tilde{m}) $ be a model graph such that
\begin{equation*}
\tilde{m}(x) =
\bar{m}(r) \text{ if } |x|=r
\end{equation*}
and
\begin{equation*}
\tilde{b}(x,y)= \begin{cases}
\bar{b}(r,r+1) & \text{ if } |x|=r,|y|=r+1 \text{ and } |x-y|=1 \\
0 & \text{ otherwise}.
\end{cases}
\end{equation*}
It is clear that, for $ r \geq 1 $, $ \bar{G} $ and $ \tilde{G} $ have the same curvature growth, but clearly $ \tilde{G} $ has twice the volume.
\end{exmp}

The above example shows that a volume comparison result cannot hold assuming that the average inner and outer curvatures satisfy the appropriate inequalities only for sufficiently large values of $r$. The last result of this section shows that this assumption is enough to control  the volume growth.

We begin with a lemma which  relates the ratio of the volume of consecutive spheres to that of the inner and outer average curvatures.

\begin{lem}\label{vol growth}
Let $ G=(V,b,m) $ be a weighted graph and $ x_0 \in V $ be a fixed vertex. Denote, as usual $ S_r = S_r(x_0) $. Then
\begin{equation*}
m(S_{r+1}) \mathcal{A}k_-(r+1) = m (S_r) \mathcal{A}k_+(r)
\end{equation*}
\end{lem}

\begin{proof}
The identity $ m(S_{r+1}) k_-(r+1) = m(S_r) k_+(r) $ is well known in the case of model graphs and, therefore, of birth-death chains (see, \cite{KLW}). The result follows by simply considering the birth-death chain associated with $ G $.
\end{proof}

\begin{defn}\label{strongcurvout}
Given two graphs $ G_1 $ and $ G_2 $, we say that $ G_1 $ has stronger average curvature growth outside of a finite set than $ G_2 $ if inequalities in Definition \ref{strongercurvnew} are satisfied for every $ r \geq R $, $ R > 0 $.
\end{defn}

\begin{thm}
\label{asympcomp}
Let $ G_1 $ and $ G_2 $ be two graphs. If $ G_1 $ has stronger average curvature growth outside of a finite set than $ G_2 $, then there exists $ C > 0 $ such that $ C m_1(S_r) \geq m_2(S_r) $ for every $ r \geq 0 $.
\end{thm}

\begin{proof}
Clearly, there exists $ C > 0 $ such that $ C m_1(S_r) \geq m_2(S_r) $ for all $ 0 \leq r \leq R $.\\
Using Lemma \ref{vol growth} and an easy inductive argument we have that, for every $ r \geq R $,
\begin{equation*}
C m_1(S_{r+1}) = C m_1(S_r) \frac{\mathcal{A}k_+^1(r)}{\mathcal{A}k_-^1(r+1)} \geq m_2(S_r) \frac{\mathcal{A}k_+^2(r)}{\mathcal{A}k_-^2(r+1)} = m_2(S_{r+1}),
\end{equation*}
completing the proof.
\end{proof}

\subsection{Ollivier curvature}
\label{OllCurv}

In this section we study some relationships between Ollivier curvature, inner and outer curvatures and volume growth. To do so, we will use expression \eqref{Oll curv} for the Ollivier curvature, which, as shown in  \cite{MW}, allows to compute the Ollivier curvature for a birth-death chain. This result can be adapted to the more general setting of model graphs using the notion of sphere curvatures of a graph that we define below following \cite{MW}. This is not surprising considering that, as we have seen so far, when dealing with curvatures we were always able to obtain comparison results by reducing the investigation to birth-death chains. Significant differences with respect to this situation will be considered at the end of this subsection.\\
We start this subsection with a definition and a couple of results, which will be useful in our discussion. They  are taken from \cite{MW} to which we refer for the proofs.

\begin{defn}
Let $ G=(V,b,m) $ be a graph and $ x_0 \in V $ be a fixed vertex. With the usual abuse of notation we write $ S_r= S_r(x_0) $. For every $ r \geq 1 $ we define the sphere curvatures $ k(r) $ with respect to $ x_0 $ as
\begin{equation*}
k(r) := \min_{y \in S_r} \max_{x \in S_{r-1} \atop x \sim y} k(x,y),
\end{equation*}
where $ k(x,y) $ is the Ollivier curvature as defined in \eqref{Oll curv}.
\end{defn}

\begin{prop}\label{Oll curv birth}(Theorem 2.10 in \cite{MW})
Let $ G=(\mathbb{N}_0,b,m) $ be a birth-death chain and let $ f(r):=d(0,r)=r $. Then, for $ 0 \leq r < R $,
\begin{equation*}
\begin{split}
k(r,R) & = \nabla_{rR} \Delta f = \frac{\Delta f(R)-\Delta f(r)}{R-r} \\
& = \frac{b\left(R,R-1\right)-b\left(R,R+1\right)}{\left( R-r \right)m(R)}- \frac{b\left(r,r-1\right)-b\left(r,r+1\right)}{\left(R-r\right)m(r)},
\end{split}
\end{equation*}
where by convention we set $ b\left(r,r-1\right) := 0 $ if $ r = 0 $.
\end{prop}

\begin{prop}\label{Oll curv comp}(Corollary 4.8 in \cite{MW})
Let $ G=(V,b,m) $ be a graph, $ x_0 \in V $ be a fixed vertex and $ k(r) $ be the sphere curvatures with respect to $ x_0 $. Let $ \bar{G}= \left( \mathbb{N}_0, \bar{b},\bar{m} \right) $ be its associated birth-death chain with root vertex $ o $ and sphere curvatures $ \bar{k}(r) = \bar{k}(r-1,r) $. Then
\begin{equation*}
\sum_{r=1}^{R} \bar{k}(r) \geq \sum_{r=1}^{R} k(r)
\end{equation*}
for all $ R \geq 1 $.
\end{prop}

Using these notation we have the following proposition. It is a convenient reformulation of \cite[Theorem 4.4]{MW} and we provide a direct proof.

\begin{prop}\label{curv comp}
Let $ \tilde{G} $ and $ G $ be two birth-death chain and $ \tilde{k}(r) = \tilde{k}(r-1,r) $, $ k(r) = k(r-1,r) $ their respective sphere curvatures for all $ r \geq 1 $. Suppose that $ k_+(0) = \tilde{k}_+(0) $. Then the following are equivalent:
\begin{itemize}
\item[(i)] $ \sum_{r=1}^{R} \tilde{k}(r) \leq \sum_{r=1}^{R} k(r) $ for all $ R \geq 1 $,
\item[(ii)] $ \tilde{k}_+(R) - \tilde{k}_-(R) \geq k_+(R) - k_-(R) $ for all $ R \geq 1 $.
\end{itemize}
\end{prop}

\begin{proof}
In order to prove the proposition, note that by Proposition \ref{Oll curv birth} the Ollivier sphere curvatures for a birth-death chain can be computed as
\begin{equation*}
k(r) = k(r-1,r) = \Delta f (r) - \Delta f(r-1) = k_-(r) - k_+(r) + k_+(r-1) - k_-(r-1)
\end{equation*}
for all $ r \geq 1 $, with $ f(r) = d(0,r) = r $ and the convention that $ k_-(0) = 0 $.\\

Case $ R = 1 $: the condition $ \tilde{k}(1) \leq k(1) $ is equivalent to
\begin{equation*}
\tilde{k}_+(0) - \tilde{k}_+(1) + \tilde{k}_-(1) \leq k_+(0) - k_+(1) + k_-(1),
\end{equation*}
which in turn is equivalent to $(ii)$ for $ R = 1 $.\\

Case $ R > 1 $: We define $ t(r) := k_+(r) - k_-(r) $ and similarly $ \tilde{t}(r) $ for all $ r \geq 0 $. Then we have
\begin{equation*}
\begin{split}
\sum_{r=1}^{R} \tilde{k}(r) & = \sum_{r=1}^{R} (\tilde{t}(r-1) -\tilde{t}(r)) \\
& = \tilde{t}(0) - \tilde{t}(R) \\
& = \tilde{k}_+(0) -\tilde{t}(R)
\end{split}
\end{equation*}
and
\begin{equation*}
\sum_{r=1}^{R} k(r) = k_+(0) - t(R),
\end{equation*}
so that
\begin{equation*}
\sum_{r=1}^{R} \tilde{k}(r) \leq \sum_{r=1}^{R} k(r)
\end{equation*}
is equivalent to
\begin{equation*}
\tilde{t}(R) = \tilde{k}_+(R) -\tilde{k}_-(R) \geq k_+(R) -k_-(R) = t(R),
\end{equation*}
completing the proof.
\end{proof}

\begin{remark}
By the above proposition, under the normalization assumption $ \tilde{k}_+(0) = k_+(0) $, we have that $ \tilde{k}(r) \leq k(r) $ for all $ r \geq 1 $ implies condition (ii) in the statement.
Further, the proposition tells us that condition (i) is not enough to imply a volume comparison result for birth-death chains, and hence for general graphs too (as we showed in Example \ref{exmp 1}).
\end{remark}

We now use Proposition \ref{Oll curv comp} and Proposition \ref{curv comp} to prove Theorem \ref{compcurv} below, which extends Proposition \ref{curv comp} to the most general possible situation. Note that the first half of the statement below generalizes \cite[Corollary 4.8]{MW}. This, in turn, will lead us to find an easy way to compute sphere curvatures on model graphs, extending Theorem 2.10 of \cite{MW}.

\begin{thm}\label{compcurv}
Let $ \tilde{G} $ be a birth-death chain and $ G $ a graph such that $ \tilde{k}_+(0) = k_+(0) $.
\begin{itemize}
\item[i)] If
\begin{equation}
\label{ineq}
\tilde{k}_+(R)-\tilde{k}_-(R) \leq k_+(x) - k_-(x) \text{ for all } x \in S_R, \text{ for all } R,
\end{equation}
then
\begin{equation*}
\sum_{r=1}^{R} \tilde{k}(r) \geq \sum_{r=1}^{R} k(r) \text{ for all } R \geq 1.
\end{equation*}
\item[ii)]If
\begin{equation*}
\sum_{r=1}^{R} \tilde{k}(r) \leq \sum_{r=1}^{R} k(r) \text{ for all } R \geq 1
\end{equation*}
then
\begin{equation*}
\tilde{k}_+(R)-\tilde{k}_-(R) \geq \min_{x \in S_R} \left( k_+(x) - k_-(x) \right) \text{ for all } R.
\end{equation*}
\end{itemize}
\end{thm}

\begin{proof}
Let $ \bar{G} $ be the birth-death chain associated to $ G $ and let $ \bar{k}_+ $ and $ \bar{k}_- $ denote its outer and inner curvatures respectively.\\
By integrating over $ S_R $ with respect to $ m $  inequality \eqref{ineq} and dividing by $ m(S_R) $, we get that
\begin{equation*}
\tilde{k}_+(R) - \tilde{k}_-(R) \leq \bar{k}_+(R) - \bar{k}_-(R) \text{ for all } R.
\end{equation*}
Theorem \ref{curv comp} combined with Proposition \ref{Oll curv comp} implies that
\begin{equation*}
\sum_{r=1}^{R} \tilde{k}(r) \geq \sum_{r=1}^{R} \bar{k}(r) \geq \sum_{r=1}^{R} k(r),
\end{equation*}
which is the first part of the theorem.\\
For the second part, using the hypothesis and Proposition \ref{Oll curv comp}, we immediately get that
\begin{equation*}
\sum_{r=1}^{R} \tilde{k}(r) \leq \sum_{r=1}^{R} \bar{k}(r),
\end{equation*}
which is equivalent, using Theorem \ref{curv comp}, to
\begin{equation*}
\tilde{k}_+(R) -\tilde{k}_-(R) \geq \bar{k}_+(R) -\bar{k}_-(R),
\end{equation*}
yielding, clearly, the desired conclusion.
\end{proof}

\begin{remark}
We note that, if $ G $ is a model graph and $ \tilde{G} $ is a birth-death chain, since $ k_+ $ and $ k_- $ are spherically symmetric functions, using the above theorem, the condition $ k_+(R) - k_-(R) \leq \tilde{k}_+(R) - \tilde{k}_-(R) $ is equivalent to $ \sum_{r=1}^{R} k(r) \geq \sum_{r=1}^{R} \tilde{k}(r) $. In particular, applying this consideration to the birth-death chain $ \tilde{G} $ associated to $ G $ we obtain the following result.
\end{remark}

\begin{thm}
Let $ G $ be a model graph and $ \tilde{G} $ its associated birth-death chain. Then
\begin{equation*}
\sum_{r=1}^{R} k(r) = \sum_{r=1}^{R} \tilde{k}(r),
\end{equation*}
for all $ R \geq 1 $, so that
\begin{equation*}
k(r) = \min_{y \in S_r} \max_{x \in S_{r-1} \atop x \sim y} k(x,y) = \tilde{k}(r)
\end{equation*}
for all $ r \geq 1 $.
\end{thm}

\begin{remark}
Note that the previous result is trivial in the case of spherically symmetric graphs, i.e. graphs such that, for every $ x, x' \in S_r $, there exists an automorphism of weighted graphs which sends $ x $ to $ x' $. Indeed, in this case, using the symmetry of the graph and  \eqref{Oll curv}, it follows   that $k(x,y)=k(x',y')=k(r)$ for every fixed $ x,x' \in S_{r-1} $ and  $y, y '\in S_r $ with $x\sim y$, $x'\sim y'$.
However the last equality is not true on general model graphs as we show in the following example. It would be interesting to investigate if there are other situations in which this  equality holds. 
\end{remark}

\begin{exmp}
Consider the graph
\begin{figure}[h]
	\begin {center}
	\begin {tikzpicture}[-latex,auto,node distance =1.5 cm and 1cm,on grid,semithick,state/.style={circle,top color =white,draw,minimum width =1 cm}]
		
	\node[state] (w) at (1, 10) {$w$};
	\node[state] (x') at (3, 12) {$x'$};
	\node[state] (x) at (3, 8) {$x$};
	\node[state] (y') at (6,12)  {$y'$};
	\node[state] (y) at (6,8) {$y$};
	\node[state] (z) at (9,8) {$z$};
	\node[state] (z') at (9,12) {$z'$};

	\path (w) edge node[] {$ 1 $} (x');
	\path (w) edge node[] {$ 1 $} (x);
	\path (x') edge node[] {$ 2 $} (y');
	\path (x) edge node[] {$ 1 $} (y);
	\path (y) edge node[] {$ 1 $} (z);
	\path (y') edge node[] {$ 3 $} (z');
	\path (x) edge node[] {$ 1 $} (y');

	\path (y') edge node[] {} (x);
	\path (x') edge node[] {} (w);
	\path (x) edge node[] {} (w);
	\path (y') edge node[] {} (x');
	\path (y) edge node[] {} (x);
	\path (z) edge node[] {} (y);
	\path (z') edge node[] {} (y');
	{};
\end{tikzpicture}
\caption{A non-spherically symmetric model graph.}
\end{center}
\end{figure}
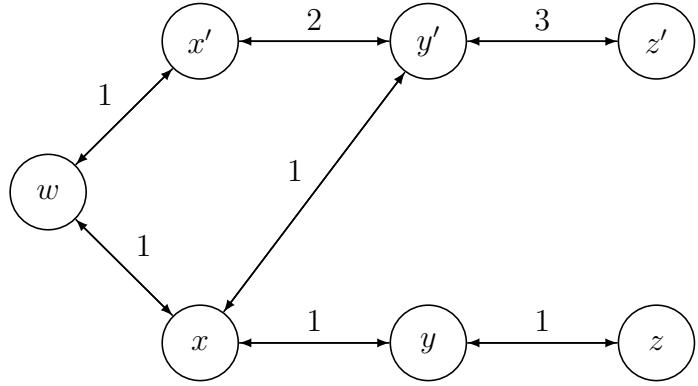

with edge weight $ b $ as in figure and measure $ m $ defined as
\begin{equation*}
m(t) = \begin{cases}
3 & \text{ if } t=y',z' \\
1 & \text{ otherwise.}
\end{cases}
\end{equation*}
By direct computation we see that the graph is model and that $ k(x,y) = -1 $. Indeed, a function which minimizes $ \Delta f(y) - \Delta f(x) $, subject to the conditions that $ f $ is $ \mathbb{Z}$-valued, Lip$(1)$ and $ f(y)-f(x)=1$, is given by 
\begin{equation*}
f(w)=f(y')=-1, \, f(x)=0, \, f(y)=1, \, f(z)=2.
\end{equation*}
This actually also follows from Example 2.3 in \cite{MW}.
On the other hand, $ k(x',y') = 1 $ since, for every $ \mathbb{Z} $-valued function $ g \in $ Lip$(1)$ with $ g(y')=1, \, g(x') = 0 $, we have
\begin{equation*}
\Delta g(y')-\Delta g(x')= 4 - g(z') + g(w) -\frac{1}{3} g(x),
\end{equation*}
which is minimized by choosing
$$
g(w)=-1, \, g(x)=0, \, g(z')=2.
$$
\end{exmp}

We end this subsection showing that, in general, even if two graphs have the same Ollivier curvature $ k(r) $ for every $ r $ it is not possible to conclude that they have the same volume growth. This is somehow interesting, being a sign that, even for birth-death chains, the Ollivier curvature is not capable of controlling the volume of spheres (see Proposition 3.5 and Remark 3.6 in \cite{HMW} for similar results).

\begin{exmp}\label{example curv Oll}
We consider again the unweighted birth-death chain $ G $ of Example \ref{exmp 1}. We want to construct a birth-death chain $ G' $ with $ k'_+(0)= 1 $, $ m'(0)=1 $ and such that $ k'(r) = k(r) $ and $ m'(r) \geq m(r) $ for every $ r $.
\begin{equation*}
k'(1) = k(1)
\end{equation*}
is equivalent to
\begin{equation*}
k'_+(0) - k'_+(1) + k'_-(1) = 1,
\end{equation*}
that is
\begin{equation*}
k'_-(1) = k'_+(1).
\end{equation*}
On the other hand we want $ m'(1) \geq m(1) $, which is equivalent, using Lemma \ref{vol growth}, to
\begin{equation*}
\frac{k'_+(0)}{k'_-(1)} \geq 1,
\end{equation*}
that is
\begin{equation*}
k'_+(0) \geq k'_-(1).
\end{equation*}
$ m'(2) \geq m(2) $ is equivalent to
\begin{equation*}
m'(1) \cdot \frac{k'_+(1)}{k'_-(2)} \geq 1,
\end{equation*}
which is implied, since $ m'(1) \geq 1 $, by $ k'_+(1) \geq k'_-(2) $. Now
\begin{equation*}
k'(2) = k(2)
\end{equation*}
is equivalent to
\begin{equation*}
k'_+(1)-k'_-(1) +k'_-(2) - k'_+(2) = 0.
\end{equation*}
Since $ k'_-(1) = k'_+(1) $ it is obvious that $ k'_-(2) - k'_+(2) = k'_-(1) -k'_+(1) = 0 $. So far we then have
\begin{equation*}
1=k'_+(0) \geq k'_-(1) = k'_+(1) \geq k'_-(2) = k'_+(2).
\end{equation*}
We then define $ k'_+(r) $ and $ k'_-(r) $ such that
\begin{equation*}
1=k'_+(0) \geq k'_-(1) = k'_+(1) \geq k'_-(2) = k'_+(2) \geq ...
\end{equation*}
and the graph $ G' $ with $ k'_+(r) $ and $ k'_-(r) $ as inner and outer curvatures and $ m'(0) = 1 $, which clearly exists. By the above discussion we have that $ G' $ has the same Ollivier curvatures as $ G $ and faster volume growth.
\end{exmp}

 We conclude the paper remarking that, while we found a satisfactory relationship between inner/outer curvature and volume growth, there seem to be no connection between volume growth and the Ollivier Ricci curvature alone. It would be interesting to investigate under which additional conditions,  assumptions  on the Ollivier curvature lead to comparison results for volumes. Probably something can be said on finite volume graphs with positive Ollivier curvature, see \cite{Pae} and \cite[Theorem 4.19] {MW}. There is yet another widely used notion of curvature, the Bakry-Emery curvature, which is not considered in this paper and it would be interesting to investigate its connections with volume growth.  We note, in this respect, that in \cite{HM} the authors obtain  volume comparison  and  volume doubling results for a class of linear graphs under Bakry-Emery curvature condition. Using parabolic methods, volume doubling results under Bakry-Emery curvature conditions are also obtained for general graphs in the very recent  \cite{HLLY} and \cite{M}. We plan to explore these and other related questions in a  forthcoming paper.

\section*{Acknowledgements}

The authors are grateful to R. K. Wojciechowski for a very careful reading and for his many comments, suggestions and bibliographical references which improved the presentation of the paper.

\end{document}